\documentclass[11pt]{article}
\usepackage{amsmath}
\usepackage{amssymb}
\usepackage{amsthm}
\usepackage[usenames]{color}
\usepackage{amscd}
\usepackage{indentfirst}

\usepackage[colorlinks=true,linkcolor=blue,filecolor=red,
citecolor=webgreen]{hyperref}
\definecolor{webgreen}{rgb}{0,.5,0}

\def\N{{\Bbb N}}
\def\Z{{\Bbb Z}}

\def\C{{\Bbb C}}

\def\1{{\bf 1}}

\def\id{\operatorname{id}}
\def\lcm{\operatorname{lcm}}

\def\gcud{\operatorname{gcud}}
\def\SL{\operatorname{SL}}

\def\ds{\displaystyle}
\def\pont{$\bullet$ }

\def\opdir{\operatorname{\Psi_{dir}}}
\def\opunit{\operatorname{\Psi_{unit}}}
\def\opgcd{\operatorname{\Psi_{gcd}}}
\def\oplcm{\operatorname{\Psi_{lcm}}}
\def\opbinom{\operatorname{\Psi_{binom}}}

\newtheorem{theorem}{Theorem}

\newtheorem{corollary}{Corollary}

\newtheorem{proposition}[theorem]{Proposition}
\newtheorem{remark}{Remark}
\newtheorem{example}{Example}

\begin{document}

\title{{\bf Multiplicative Arithmetic Functions of Several Variables: A Survey}}
\author{L\'aszl\'o T\'oth}
\date{}
\maketitle

\centerline{in vol. Mathematics Without Boundaries}
\centerline{Surveys in Pure Mathematics}
\centerline{T. M. Rassias, P. M. Pardalos (eds.), Springer, 2014, pp. 483--514}

\abstract{We survey general properties of multiplicative
arithmetic functions of several variables and related
convolutions, including the Dirichlet convolution and the unitary
convolution. We introduce and investigate a new convolution, called
gcd convolution. We define and study the convolutes of arithmetic
functions of several variables, according to the different types of
convolutions. We discuss the multiple Dirichlet series and Bell series and present certain arithmetic
and asymptotic results of some special multiplicative functions arising
from problems in number theory, group theory and combinatorics. We
give a new proof to obtain the asymptotic density of the set of
ordered $r$-tuples of positive integers with pairwise relatively prime components and consider a similar question
related to unitary divisors.}

\medskip



\tableofcontents


\section{Introduction}

Multiplicative arithmetic functions of a single variable are very
well known in the literature. Their various properties were
investigated by several authors and they represent an important
research topic up to now. Less known are multiplicative arithmetic
functions of several variables of which detailed study was carried
out by {\sc R.~Vaidyanathaswamy} \cite{Vai1931} more than eighty
years ago. Since then many, sometimes scattered results for the
several variables case were published in papers and monographs, and
some authors of them were not aware of the paper \cite{Vai1931}. In
fact, were are two different notions of multiplicative functions of
several variables, used in the last decades, both reducing to the
usual multiplicativity in the one variable case. For the other
concept we use the term firmly multiplicative function.

In this paper we survey general properties of multiplicative
arithmetic functions of several variables and related
convolutions, including the Dirichlet convolution and the unitary
convolution. We introduce and investigate a new convolution, called
gcd convolution. We define and study the convolutes of arithmetic
functions of several variables, according to the different types of
convolutions. The concept of the convolute of a function with
respect to the Dirichlet convolution was introduced by {\sc
R.~Vaidyanathaswamy} \cite{Vai1931}. We also discuss the multiple
Dirichlet series and Bell series. We present certain arithmetic and
asymptotic results of some special multiplicative functions arising
from problems in number theory, group theory and combinatorics. We
give a new proof to obtain the asymptotic density of the set of
ordered $r$-tuples of positive integers with pairwise relatively prime components.
Furthermore, we consider a similar question, namely the asymptotic density of the set of
ordered $r$-tuples with pairwise {\sl unitary relatively prime} components, that is, the greatest common unitary
divisor of each two distinct components is $1$.

For general properties of (multiplicative) arithmetic functions of a
single variable see, e.g., the books of {\sc T.~M.~Apostol}
\cite{Apo1976}, {\sc G.~H.~Hardy, E.~M.~Wright} \cite{HarWri2008},
{\sc P.~J.~Mc.Carthy} \cite{McC1986}, {\sc W.~Schwarz, J.~Spilker}
\cite{SchSpi1994} and {\sc R.~Sivaramakrishnan} \cite{Siv1989}. For
algebraic properties of the ring of arithmetic functions of a single
variable with the Dirichlet convolution we refer to {\sc
H.~N.~Shapiro} \cite{Sha1972}. Incidence algebras and semilattice
algebras concerning arithmetic functions of a single variable were
investigated by {\sc D.~A.~Smith} \cite{Smi1972}. For properties of certain
subgroups of the group of multiplicative arithmetic functions of a single variable under the Dirichlet convolution
we refer to the papers by {\sc T.~B.~Carroll, A.~A.~Gioia} \cite{CarGio1975}, {\sc P.-O.~Dehaye} \cite{Deh2002},
{\sc J.~E.~Delany} \cite{Del2005}, {\sc T.~MacHenry} \cite{MacH1999} and {\sc R.~W.~Ryden} \cite{Ryd1973}. Algebraical and
topological properties of the ring of arithmetic functions of a single variable with the unitary convolution were given by {\sc
J.~Snellman} \cite{Sne2003,Sne2004}. See also {\sc J.~S\'andor,
B.~Crstici} \cite[Sect.\ 2.2]{SanCrs2004} and {\sc H.~Scheid}
\cite{Sche1969}.


\section{Notations}

Throughout the paper we use the following notations.

General notations:

\pont $\N=\{1,2,\ldots\}$, $\N_0=\{0,1,2,\ldots\}$,

\pont the prime power factorization of $n\in \N$ is $n=\prod_p
p^{\nu_p(n)}$, the product being over the primes $p$, where all but
a finite number of the exponents $\nu_p(n)$ are zero,

\pont $d\mid\mid n$ means that $d$ is a unitary divisor of $n$, i.e., $d\mid n$ and $\gcd(d,n/d)=1$,

\pont $\gcud(n_1,\ldots,n_k)$ denotes the greatest common unitary divisor of $n_1,\ldots,n_k\in \N$,

\pont $\Z_n=\Z/n\Z$ is the additive group of residue classes modulo
$n$,

\pont $\zeta$ is the Riemann zeta function,

\pont $\gamma$ is Euler's constant.

Arithmetic functions of a single variable:

\pont $\delta$ is the arithmetic function given by $\delta(1)=0$ and
$\delta(n)=0$ for $n>1$,

\pont $\id$ is the function $\id(n)=n$ ($n\in \N$),

\pont $\phi_k$ is the Jordan function of order $k$ given by
$\phi_k(n)=n^k\prod_{p\mid n} (1-1/p^k)$ ($k\in \C$),

\pont $\phi=\phi_1$ is Euler's totient function,

\pont $\psi$ is the Dedekind function given by
$\psi(n)=n\prod_{p\mid n} (1+1/p)$,

\pont $\mu$ is the M\"obius function,

\pont $\tau_k$ is the Piltz divisor function of order $k$, $\tau_k(n)$
representing the number of ways of expressing $n$ as a product of
$k$ factors,

\pont $\tau(n)=\tau_2(n)$ is the number of divisors of $n$,

\pont $\sigma_k(n)=\sum_{d\mid n} d^k$ ($k\in \C$),

\pont $\sigma(n)=\sigma_1(n)$ is the sum of divisors of $n$,

\pont $\omega(n)=\# \{p: \nu_p(n)\ne 0\}$ stands for the number of
distinct prime divisors of $n$,

\pont $\mu^{\times}(n)=(-1)^{\omega(n)}$

\pont $\Omega(n)=\sum_p \nu_p(n)$ is the number of prime power
divisors of $n$,

\pont $\lambda(n)=(-1)^{\Omega(n)}$ is the Liouville function,

\pont $\xi(n)=\prod_p \nu_p(n)!$,

\pont $c_n(k)=\sum_{1\le q \le n, \gcd(q,n)=1} \exp(2\pi iqk/n)$
($n,k\in \N$) is the Ramanujan sum, which can be viewed as a
function of two variables.

Arithmetic functions of several variables:

\pont ${\cal A}_r$ is the set of arithmetic functions of $r$ variables ($r\in \N$), i.e.,
of functions $f\colon \N^r \to \C$,

\pont ${\cal A}_r^{(1)} = \{f\in {\cal A}_r: f(1,\ldots,1)\ne 0\}$,

\pont $\1_r$ is the constant $1$ function in ${\cal A}_r$, i.e.,
$\1_r(n_1,\ldots,n_r)=1$ for every $n_1,\ldots,n_r\in \N$,

\pont $\delta_r(n_1,\ldots,n_r)=\delta(n_1)\cdots \delta(n_r)$, that
is $\delta_r(1,\ldots,1)=1$ and $\delta_r(n_1,\ldots,n_r)=0$
for $n_1\cdots n_r>1$,

\pont for $f\in {\cal A}_r$ the function $\overline{f}\in {\cal
A}_1$ is given by $\overline{f}(n)= f(n,\ldots,n)$ for every $n\in
\N$.

Other notations will be fixed inside the paper.


\section{Multiplicative functions of several variables}

In what follows we discuss the notions of multiplicative, firmly
multiplicative and completely multiplicative functions. We point out
that properties of firmly and completely multiplicative functions of
several variables reduce to those of multiplicative, respectively
completely multiplicative functions of a single variable.
Multiplicative functions can not be reduced to functions of a single
variable. We also present examples of such functions.

\subsection{Multiplicative functions}

A function $f\in {\cal A}_r$ is said to be multiplicative if it is
not identically zero and
\begin{equation*}
f(m_1n_1,\ldots,m_rn_r)= f(m_1,\ldots,m_r) f(n_1,\ldots,n_r)
\end{equation*}
holds for any $m_1,\ldots,m_r,n_1,\ldots,n_r\in \N$ such that
$\gcd(m_1\cdots m_r,n_1\cdots n_r)=1$.

If $f$ is multiplicative, then it is determined by the values
$f(p^{\nu_1},\ldots,p^{\nu_r})$, where $p$ is prime and
$\nu_1,\ldots,\nu_r\in \N_0$. More exactly, $f(1,\ldots,1)=1$ and
for any $n_1,\ldots,n_r\in \N$,
\begin{equation*}
f(n_1,\ldots,n_r)= \prod_p f(p^{\nu_p(n_1)}, \ldots,p^{\nu_p(n_r)}).
\end{equation*}

If $r=1$, i.e., in the case of functions of a single variable we
reobtain the familiar notion of multiplicativity: $f\in {\cal A}_1$
is multiplicative if it is not identically zero and $f(mn)=
f(m)f(n)$ for every $m,n\in \N$ such that $\gcd(m,n)=1$.

Let ${\cal M}_r$ denote the set of multiplicative functions in $r$
variables.

\subsection{Firmly multiplicative functions}

We call a function $f\in {\cal A}_r$ {\sl firmly multiplicative} (following
{\sc P.~Haukkanen} \cite{Hau}) if it is not identically zero and
\begin{equation*}
f(m_1n_1,\ldots,m_rn_r)= f(m_1,\ldots,m_r) f(n_1,\ldots,n_r)
\end{equation*}
holds for any $m_1,\ldots,m_r,n_1,\ldots,n_r\in \N$ such that
$\gcd(m_1,n_1)=\ldots = \gcd(m_r,n_r)=1$. Let ${\cal F}_r$ denote
the set of firmly multiplicative functions in $r$ variables.

A firmly multiplicative function is completely determined by its
values at \newline $(1, \ldots, 1, p^\nu, 1, \ldots,1)$, where p runs through
the primes and $\nu \in \N_0$. More exactly, $f(1,\ldots,1)=1$ and
for any $n_1,\ldots,n_r\in \N$,
\begin{equation*}
f(n_1,\ldots,n_r)= \prod_p \left( f(p^{\nu_p(n_1)},1,\ldots,1)\cdots
f(1,\ldots,1,p^{\nu_p(n_r)}) \right).
\end{equation*}

If a function $f\in {\cal A}_r$ is firmly multiplicative, then it is
multiplicative. Also, if $f\in {\cal F}_r$, then
$f(n_1,\ldots,n_r)=f_1(n_1,1,\ldots, 1)\cdots f_r(1,\ldots,1,n_r)$
for every $n_1,\ldots,n_r\in \N$. This immediately gives the
following property:

\begin{proposition} \label{Prop_firmly_repr} A function $f\in {\cal A}_r$ is firmly multiplicative
if and only if there exist multiplicative functions $f_1,\ldots,
f_r\in {\cal M}_1$ (each of a single variable) such that
$f(n_1,\ldots,n_r)=f_1(n_1)\cdots f_r(n_r)$ for every
$n_1,\ldots,n_r\in \N$. In this case $f_1(n)= f(n,1, \ldots, 1)$, ...,
$f_r(n)=f(1,\ldots,1,n)$ for every $n\in \N$.
\end{proposition}

In the case of functions of a single variable the notion of firmly
multiplicative function reduces to that of multiplicative function.
For $r>1$ the concepts of multiplicative and firmly multiplicative
functions are different.

\subsection{Completely multiplicative functions}

A function $f\in {\cal A}_r$ is called completely multiplicative if
it is not identically zero and
\begin{equation*}
f(m_1n_1,\ldots,m_rn_r)= f(m_1,\ldots,m_r) f(n_1,\ldots,n_r)
\end{equation*}
holds for any $m_1,\ldots,m_r,n_1,\ldots,n_r\in \N$. Note that {\sc R.~Vaidyanathaswamy} \cite{Vai1931}
used for such a function the term 'linear function'.

Let ${\cal C}_r$ denote the set of completely multiplicative
functions in $r$ variables. If $f\in {\cal C}_r$, then it is
determined by its values at $(1, \ldots, 1, p, 1, \ldots,1)$, where p runs through the
primes. More exactly, $f(1,\ldots,1)=1$ and for any
$n_1,\ldots,n_r\in \N$,
\begin{equation*}
f(n_1,\ldots,n_r)= \prod_p \left( f(p,1,\ldots,1)^{\nu_p(n_1)}
\cdots f(1,\ldots,1,p)^{\nu_p(n_r)} \right).
\end{equation*}

In the case of functions of a single variable we reobtain the
familiar notion of completely multiplicative function: $f\in {\cal
A}_1$ is completely multiplicative if it is not identically zero and
$f(mn)= f(m)f(n)$ for every $m,n\in \N$.

It is clear that if a function $f\in {\cal A}_r$ is completely
multiplicative, then it is firmly multiplicative. Also, similar to
Proposition \ref{Prop_firmly_repr}:

\begin{proposition} \label{Prop_compl_repr} A function $f\in {\cal A}_r$ is
completely multiplicative if and only
if there exist completely multiplicative functions $f_1,\ldots,
f_r\in {\cal C}_1$ (each of a single variable) such that
$f(n_1,\ldots,n_r)=f_1(n_1)\cdots f_r(n_r)$ for every
$n_1,\ldots,n_r\in \N$. In this case $f_1(n)=f(n,1,\ldots,1)$, ...,
$f_r(n)=f(1,\ldots,1,n)$ for every $n\in \N$.
\end{proposition}

\subsection{Examples} \label{subsection_examples}

The functions $(n_1,\ldots,n_r) \mapsto \gcd(n_1,\ldots,n_r)$ and
$(n_1,\ldots,n_r) \mapsto \lcm(n_1,\ldots,n_r)$ are multiplicative
for every $r\in \N$, but not firmly multiplicative for $r\ge 2$.

The functions $(n_1,\ldots,n_r) \mapsto \tau(n_1)\cdots \tau(n_r)$,
$(n_1,n_2)\mapsto \tau(n_1)\sigma(n_2)$ are firmly multiplicative,
but not completely multiplicative.

The functions $(n_1,\ldots,n_r) \mapsto n_1\cdots n_r$,
$(n_1,n_2)\mapsto n_1 \lambda(n_2)$ are completely multiplicative.

According to Propositions \ref{Prop_firmly_repr} and \ref{Prop_compl_repr} firmly multiplicative and completely
multiplicative functions reduce to multiplicative, respectively completely multiplicative functions of a single variable.
There is no similar characterization for multiplicative functions of several variables.

Let $h\in {\cal M}_1$. Then the functions $(n_1,\ldots,n_r)
\mapsto h(\gcd(n_1,\ldots,n_r))$, $(n_1,\ldots,n_r) \mapsto$ \\ $ h(\lcm(n_1,\ldots,n_r))$ are
multiplicative. The product and the quotient of (nonvanishing)
multiplicative functions are multiplicative.

If $f\in {\cal M}_r$ is multiplicative and we fix one (or more, say
$s$) variables, then the resulting function of $r-1$ (or $r-s$)
variables is not necessary multiplicative. For example, $(k,n)\mapsto
c_n(k)$ is multiplicative as a function of two variables, see Section \ref{subsection_Dirichlet_convolution},
but for a fixed $n$ the function $k\mapsto c_n(k)$ is in general not
multiplicative (it is multiplicative if and only if $\mu(n)=1$).

If $f\in {\cal M}_r$ is multiplicative, then the function
$\overline{f}$ of a single variable is multiplicative.

Other examples from number theory, group theory and combinatorics:

\begin{example} {\rm Let $N_{g_1,\ldots,g_r}(n_1,\ldots,n_r)$ denote the number
of solutions $x$ (mod $n$), with $n=\lcm(n_1,\ldots,n_r)$, of the simultaneous
congruences $g_1(x)\equiv 0$ (mod $n_1$), ..., $g_r(x)\equiv 0$ (mod
$n_r$), where $g_1,\ldots,g_r$ are polynomials with integer
coefficients. Then the function $(n_1,\ldots,n_r)\mapsto
N_{g_1,\ldots,g_r}(n_1,\ldots,n_r)$ is multiplicative. See {\sc
L.~T\'oth} \cite[Sect.\ 2]{Tot2011Menon} for a proof.}
\end{example}

\begin{example} \label{ex_varrho} {\rm For a fixed integer $r\ge 2$ let
\begin{equation} \label{def_pair_rel_prime}
\varrho(n_1,\ldots,n_r) = \begin{cases} 1, \text{ if $n_1,\ldots, n_r$ are pairwise relatively prime},\\
0, \text{ otherwise}. \end{cases}
\end{equation}

This function is multiplicative, which follows from the definition, and for every $n_1,\ldots,n_r\in \N$,
\begin{equation} \label{formula_varrho}
\varrho(n_1,\ldots,n_r) = \sum_{d_1\mid n_1, \ldots, d_r\mid n_r}
\tau(d_1\cdots d_r) \mu(n_1/d_1)\cdots \mu(n_r/d_r),
\end{equation}
cf. Section \ref{subsection_Dirichlet_convolution}.}
\end{example}

\begin{example} \label{ex_unit_varrho} {\rm For $r\ge 2$ let
\begin{equation} \label{def_pair_unit_rel_prime}
\varrho^{\times}(n_1,\ldots,n_r) = \begin{cases} 1, \text{ if $\gcud(n_i,n_j)=1$ for every $i\ne j$},\\
0, \text{ otherwise}. \end{cases}
\end{equation}

This is the characteristic function of the set of ordered $r$-tuples
$(n_1,\ldots,n_r)\in \N^r$ such that $n_1,\ldots,n_r$ are pairwise
unitary relatively prime, i.e., such that for every prime $p$ there
are no $i\ne j$ with $\nu_p(n_i)=\nu_p(n_j)\ge 1$. This function is
also multiplicative (by the definition). }
\end{example}

\begin{example} \label{ex_s_c} {\rm Consider the group $G=\Z_{n_1}\times \cdots \times \Z_{n_r}$.
Let $s(n_1,\ldots,n_r)$ and $c(n_1,\ldots,n_r)$ denote the total
number of subgroups of the group $G$ and the number of its cyclic
subgroups, respectively. Then the functions $(n_1,\ldots,n_r)
\mapsto s(n_1,\ldots,n_r)$ and $(n_1,\ldots,n_r) \mapsto
c(n_1,\ldots,n_r)$ are multiplicative. For every $n_1,\ldots,n_r\in
\N$,
\begin{equation} \label{formula_c_r}
c(n_1,\ldots,n_r) = \sum_{d_1\mid n_1, \ldots, d_r\mid n_r}
\frac{\phi(d_1)\cdots \phi(d_r)}{\phi(\lcm(d_1,\ldots,d_r))},
\end{equation}
see {\sc L.~T\'oth} \cite[Th.\ 3]{Tot2011Menon}, \cite[Th.\
1]{Tot2012Rou}. In the case $r=2$ this gives
\begin{equation*}
c(n_1,n_2) = \sum_{d_1\mid n_1, d_2\mid n_2} \phi(\gcd(d_1,d_2)).
\end{equation*}

Also,
\begin{equation} \label{formula_s}
s(n_1,n_2) = \sum_{d_1\mid n_1, d_2\mid n_2} \gcd(d_1,d_2),
\end{equation}
for every $n_1,n_2\in \N$. See {\sc M.~Hampejs, N.~Holighaus,
L.~T\'oth, C.~Wiesmeyr} \cite{HHTW2012} and {\sc M.~Hampejs,
L.~T\'oth} \cite{HamTot2013}.}
\end{example}

\begin{example} {\rm We define the sigma function of $r$ variables by
\begin{equation} \label{formula_sigma_r}
\sigma(n_1,\ldots,n_r) = \sum_{d_1\mid n_1, \ldots, d_r\mid n_r}
\gcd(d_1,\ldots,d_r)
\end{equation}
having the representation
\begin{equation}
\sigma(n_1,\ldots,n_r) = \sum_{d\mid \gcd(n_1, \ldots, n_r)} \phi(d)\tau(n_1/d)\cdots \tau(n_r/d),
\end{equation}
valid for every $n_1,\ldots,n_r\in \N$. Note that for $r=1$ this function reduces to the sum-of-divisors function and
in the case $r=2$ we have $\sigma(m,n)=s(m,n)$ given in Example \ref{ex_s_c}.

We call $(n_1,\ldots,n_r)\in \N^r$ a {\sl perfect $r$-tuple} if $\sigma(n_1,\ldots,n_r)=2\gcd(n_1,\ldots,n_r)$. If $2^r-1=p$ is a
Mersenne prime, then $(p,p,\ldots,p)$ is a perfect $r$-tuple. For example, $(3,3)$ is a perfect pair and $(7,7,7)$ is perfect triple.
We formulate as an open problem: Which are all the perfect $r$-tuples? }
\end{example}

\begin{example} {\rm The Ramanujan sum $(k,n)\mapsto c_n(k)$ having the representation
\begin{equation} \label{formula_Ramanujan}
c_n(k) = \sum_{d\mid \gcd(k,n)} d \mu(n/d)
\end{equation}
is multiplicative as a function of two variables. This property was
pointed out by {\sc K.~R.~Johnson} \cite{Joh1986}, see also Section
\ref{subsection_Dirichlet_convolution}.}
\end{example}

\begin{example} {\rm For $n_1,\ldots,n_r \in \N$ let $n:=\lcm(n_1,\ldots,n_r)$.
The function of $r$ variables
\begin{equation*}
E(n_1,\ldots,n_r)= \frac1{n}\sum_{j=1}^n c_{n_1}(j)\cdots c_{n_r}(j)
\end{equation*}
has combinatorial and topological applications, and was investigated
in the papers of {\sc V.~A. Liskovets} \cite{Lis2010} and {\sc L.~T\'oth}
\cite{Tot2012Lisko}. All values of $E(n_1,\ldots,n_r)$ are
nonnegative integers and the function $E$ is multiplicative.
Furthermore, it has the following representation (\cite[Prop.\ 3]{Tot2012Lisko}):
\begin{equation} \label{formula_E}
E(n_1,\ldots,n_r) = \sum_{d_1\mid n_1, \ldots, d_r\mid n_r}
\frac{d_1\mu(n_1/d_1)\cdots d_r\mu(n_r/d_r)}{\lcm(d_1,\ldots,d_r)},
\end{equation}
valid for every $n_1,\ldots,n_r\in \N$. See also {\sc L.~T\'oth}
\cite{Tot2012Ferrara,Toth2013Alkan} for generalizations of
the function $E$.}
\end{example}

\begin{example} {\rm Another multiplicative function, similar to $E$ is
\begin{equation*}
A(n_1,\ldots,n_r)= \frac1{n} \sum_{k=1}^n \gcd(k,n_1)\cdots
\gcd(k,n_r),
\end{equation*}
where $n_1,\ldots, n_r\in \N$ and $n:=\lcm(n_1,\ldots,n_r)$, as
above. One has for every $n_1,\ldots,n_r\in \N$,
\begin{equation} \label{formula_A}
A(n_1,\ldots,n_r) = \sum_{d_1\mid n_1, \ldots, d_r\mid n_r}
\frac{\phi(d_1)\cdots \phi(d_r)}{\lcm(d_1,\ldots,d_r)},
\end{equation}
see {\sc L.~T\'oth} \cite[Eq.\ (45)]{Tot2010GCD}, \cite[Prop.\
12]{Tot2012Lisko}.}
\end{example}

See the paper of {\sc M.~Peter} \cite{Pet1997} for properties of recurrent multiplicative arithmetical
functions of several variables.


\section{Convolutions of arithmetic functions of several variables} \label{sect_convo}

In this Section we survey the basic properties of the Dirichlet and unitary convolutions of arithmetic functions of several variables. We also
define and discuss the gcd, lcm and binomial convolutions, not given in the literature in the several variables case. We point out that the gcd
convolution reduces to the unitary convolution in the one variable case, but they are different for $r$ variables with $r>1$. For other
convolutions we refer to the papers of {\sc J.~S\'andor, A.~Bege} \cite{SanBeg2003}, {\sc E.~D.~Schwab} \cite{Schw2010} and
{\sc M.~V.~Subbarao} \cite{Sub1972}.

\subsection{Dirichlet convolution} \label{subsection_Dirichlet_convolution}

For every $r\in \N$ the set ${\cal A}_r$ of arithmetic functions of
$r$ variables is a $\C$-linear space with the usual linear
operations. With the Dirichlet convolution defined by
\begin{equation*}
(f*g)(n_1,\ldots,n_r)= \sum_{d_1\mid n_1, \ldots, d_r\mid n_r}
f(d_1,\ldots,d_r) g(n_1/d_r, \ldots, n_r/d_r)
\end{equation*}
the space ${\cal A}_r$ forms a unital commutative $\C$-algebra, the
unity being the function $\delta_r$, and $({\cal A}_r,+,*)$ is an
integral domain. Moreover, $({\cal A}_r,+,*)$ is a unique
factorization domain, as pointed out by {\sc T.~Onozuka}
\cite{Ono2013}. In the case $r=1$ this was proved by {\sc
E.~D.~Cashwell, C.~J.~Everett} \cite{CasEve1959}. The group of
invertible functions is ${\cal A}_r^{(1)}$. The inverse of $f$ will
be denoted by $f^{-1*}$. The inverse of the constant $1$ function
$\1_r$ is $\mu_r$, given by $\mu_r(n_1,\ldots,n_r)=\mu(n_1)\cdots
\mu(n_r)$ (which is firmly multiplicative, where $\mu$ is the
classical M\"obius function).

The Dirichlet convolution preserves the multiplicativity of
functions. This property, well known in the one variable case,
follows easily from the definitions. Using this fact the
multiplicativity of the functions $c(n_1,\ldots,n_r)$, $s(n_1,n_2)$,
$\sigma(n_1,\ldots,n_r)$, $E(n_1,\ldots,n_r)$ and
$A(n_1,\ldots,n_r)$ is a direct consequence of the convolutional
representations \eqref{formula_c_r}, \eqref{formula_s},
\eqref{formula_sigma_r}, \eqref{formula_E} and \eqref{formula_A},
respectively. The multiplicativity of the Ramanujan sum
$(k,n)\mapsto c_n(k)$ follows in a similar manner from
\eqref{formula_Ramanujan}, showing that $c_n(k)$ is the convolution
of the multiplicative functions $f$ and $g$ defined by $f(m,n)=m$
for $m=n$, $f(m,n)=0$ for $m\ne n$ and $g(m,n)=\mu(m)$ for every
$m,n\in \N$.

If $f,g\in {\cal F}_r$, then
\begin{equation*}
(f*g)(n_1,\ldots,n_r)=(f_1*g_1)(n_1) \cdots (f_r*g_r)(n_r)
\end{equation*}
and
\begin{equation*}
f^{-1*}(n_1,\ldots,n_r)=f_1^{-1*}(n_1) \cdots f_r^{-1*}(n_r),
\end{equation*}
with the notations of Proposition \ref{Prop_firmly_repr}, hence $f*g\in
{\cal F}_r$ and $f^{-1*}\in {\cal F}_r$. We deduce

\begin{proposition} \label{Prop_subgroup_Dir}
One has the following subgroup relations:
\begin{equation*}
({\cal F}_r,*)\le ({\cal M}_r,*)\le ({\cal A}_r^{(1)},*).
\end{equation*}
\end{proposition}

The set ${\cal C}_r$ does not form a group under the Dirchlet
convolution. If $f\in {\cal C}_r$, then $f^{-1*}=\mu_r f$ (well
known in the case $r=1$). Note that for every $f\in {\cal C}_r$ one
has $(f*f)(n_1,\ldots,n_r)=f(n_1,\ldots,n_r) \tau(n_1)\cdots
\tau(n_r)$, in particular
$(\1_r*\1_r)(n_1,\ldots,n_r)=\tau(n_1)\cdots \tau(n_r)$.

{\sc R.~Vaidyanathaswamy} \cite{Vai1931} called the Dirichlet convolution 'composition of functions'.

Other convolutional properties known in the one variable case, for
example M\"obius inversion can easily be generalized. As mentioned
in Section \ref{subsection_examples}, Example \ref{ex_varrho} the
characteristic function $\varrho$ of the set of ordered $r$-tuples
$(n_1,\ldots,n_r)\in \N^r$ such that $n_1,\ldots,n_r\in \N$ are
pairwise relatively prime is multiplicative. One has for every
$n_1,\ldots,n_r\in \N$,
\begin{equation} \label{formula_2_varrho}
\sum_{d_1\mid n_1, \ldots, d_r\mid n_r} \varrho(d_1,\ldots,d_r)= \tau(n_1\cdots n_r),
\end{equation}
since both sides are multiplicative and in the case of prime powers $n_1=p^{\nu_1},\ldots,n_r=p^{\nu_r}$ both sides of \eqref{formula_2_varrho}
are equal to $1+\nu_1+\ldots +\nu_r$. Now, M\"obius inversion gives the formula \eqref{formula_varrho}.

For further algebraic properties of the $\C$-algebra ${\cal A}_r$
and more generally, of the $R$-algebra $A_r(R)=\{f:\N^r\mapsto R\}$,
where $R$ is an integral domain and using the concept of firmly
multiplicative functions see {\sc E.~Alkan, A.~Zaharescu, M.~Zaki} \cite{AlkZahZak2005}. That paper
includes, among others, constructions of a class of derivations and
of a family of valuations on $A_r(R)$. See also {\sc P.~Haukkanen} \cite{Hau} and {\sc A.~Zaharescu, M.~Zaki}
\cite{ZahZak2007}.

\subsection{Unitary convolution}

The linear space ${\cal A}_r$ forms another unital commutative
$\C$-algebra with the unitary convolution defined by
\begin{equation*}
(f\times g)(n_1,\ldots,n_r)= \sum_{d_1\mid\mid n_1, \ldots,
d_r\mid\mid n_r} f(d_1,\ldots,d_r) g(n_1/d_1, \ldots, n_r/d_r).
\end{equation*}

Here the unity is the function $\delta_r$ again. Note that $({\cal
A}_r,+,\times)$ is not an integral domain, there exist divisors of
zero. The group of invertible functions is again ${\cal A}_r^{(1)}$.
The unitary $r$ variables M\"obius function $\mu^{\times}_r$ is
defined as the inverse of the function $\1_r$. One has
$\mu^{\times}_r(n_1,\ldots,n_r)=\mu^{\times}(n_1)\cdots
\mu^{\times}(n_r)= (-1)^{\omega(n_1)+\ldots +\omega(n_r)}$ (and it
is firmly multiplicative). Similar to Proposition
\ref{Prop_subgroup_Dir},

\begin{proposition} One has the following subgroup relations:
\begin{equation*}
({\cal F}_r,\times)\le ({\cal M}_r,\times)\le ({\cal
A}_r^{(1)},\times).
\end{equation*}
\end{proposition}

If $f\in {\cal F}_r$, then its inverse is
$f^{-1\times}=\mu^{\times}_r f$ and
\begin{equation*}
(f\times f)(n_1,\ldots,n_r)= f(n_1,\ldots,n_r) 2^{\omega(n_1)+\ldots +
\omega(n_r)},
\end{equation*}
in particular $(\1_r\times \1_r)(n_1,\ldots,n_r)=
2^{\omega(n_1)+\ldots + \omega(n_r)}$.

{\sc R.~Vaidyanathaswamy} \cite{Vai1931} used the term 'compounding
of functions' for the unitary convolution.

For further algebraic properties of the $R$-algebra
$A_r(R)=\{f:\N^r\mapsto R\}$, where $R$ is an integral domain with
respect to the unitary convolution and using the concept of firmly
multiplicative functions see {\sc E.~Alkan, A.~Zaharescu, M.~Zaki} \cite{AlkZahZak2006}.

\subsection{Gcd convolution}

We define a new convolution for functions $f,g\in {\cal A}_r$, we call it {\sl gcd convolution}, given by
\begin{equation} \label{gcd_convo_functions}
(f\odot g)(n_1,\ldots,n_r)= \sum_{\substack{d_1e_1=n_1, \ldots,
d_re_r=n_r \\ \gcd(d_1\cdots d_r,e_1\cdots e_r)=1}}
f(d_1,\ldots,d_r) g(e_1, \ldots, e_r),
\end{equation}
which is in concordance with the definition of multiplicative functions.

In the case $r=1$ the unitary and gcd convolutions are identic,
i.e., $f\times g=f\odot g$ for every $f,g\in {\cal A}_1$, but they
differ for $r>1$.

Main properties: ${\cal A}_r$ forms a unital commutative
$\C$-algebra with the gcd  convolution defined by
\eqref{gcd_convo_functions}. The unity is the function $\delta_r$ and there exist divisors of zero.
The group of invertible functions is again ${\cal A}_r^{(1)}$. Here the inverse of the
constant $1$ function is $\mu^{\odot}_r(n_1,\ldots,n_r)=
(-1)^{\omega(n_1\cdots n_r)}$. More generally, the inverse
$f^{-1\odot}$ of an arbitrary multiplicative function $f$ is given
by $f^{-1\odot}(n_1,\ldots,n_r)= (-1)^{\omega(n_1\cdots
n_r)}f(n_1,\ldots,n_r)$. Also, if $f\in {\cal M}_r$, then
\begin{equation*}
(f\odot f)(n_1,\ldots,n_r)= f(n_1,\ldots,n_r) 2^{\omega(n_1\cdots n_r)}.
\end{equation*}

\begin{proposition} One has
\begin{equation*}
({\cal M}_r,\odot)\le ({\cal A}_r^{(1)},\odot).
\end{equation*}
\end{proposition}

The set ${\cal F}_r$ does not form a group under the gcd
convolution. To see this note that $(\1_r\odot
\1_r)(n_1,\ldots,n_r)= 2^{\omega(n_1\cdots n_r)}$, but this function
is not firmly multiplicative, since $\omega(n_1\cdots n_r)= \omega(n_1)+\ldots+\omega(n_r)$ does not
hold for every $n_1,\ldots,n_r\in \N$ (cf. Proposition \ref{Prop_firmly_repr}).

\subsection{Lcm convolution}

We define the lcm convolution of functions of $r$ variables by
\begin{equation*}
(f\oplus g)(n_1,\ldots,n_r)= \sum_{\lcm(d_1,e_1)=n_1, \ldots,
\lcm(d_r,e_r)=n_r} f(d_1,\ldots,d_r) g(e_1, \ldots, e_r).
\end{equation*}

In the case $r=1$ this convolution originates by {\sc R.~D.~von Sterneck}
\cite{Ste1894}, was investigated and generalized by {\sc D.~H.~Lehmer}
\cite{Leh1931a,Leh1931b,Leh1931c}, and is also called von
Sterneck-Lehmer convolution. See also {\sc R.~G.~Buschman} \cite{Bus1999}.

Note that the lcm convolution can be expressed by the Dirichlet
convolution. More exactly,

\begin{proposition} \label{prop_lcm_convo} For every $f,g\in {\cal A}_r$,
\begin{equation} \label{lcm_convo_connection}
f\oplus g = (f*\1_r)(g*\1_r)*\mu_r.
\end{equation}
\end{proposition}

\begin{proof} Write
\begin{equation*}
\sum_{a_1\mid n_1,\ldots,a_r\mid n_r} (f\oplus g)(a_1,\ldots,a_r)
\end{equation*}
\begin{equation*}
=\sum_{\lcm(d_1,e_1)\mid n_1, \ldots, \lcm(d_r,e_r) \mid n_r}
f(d_1,\ldots,d_r) g(e_1, \ldots, e_r)
\end{equation*}
\begin{equation*}
=  \sum_{d_1\mid n_1, \ldots, d_r \mid n_r} f(d_1,\ldots,d_r)
\sum_{e_1\mid n_1,\ldots,e_r\mid n_r} g(e_1, \ldots, e_r)
\end{equation*}
\begin{equation*}
= (f*\1_r)(n_1,\ldots,n_r) (g*\1_r)(n_1,\ldots,n_r),
\end{equation*}
and by M\"obius inversion we obtain \eqref{lcm_convo_connection}.
\end{proof}

In the case $r=1$ Proposition \ref{prop_lcm_convo} is due to
{\sc R.~D.~von~Sterneck} \cite{Ste1894} and {\sc D.~H.~Lehmer} \cite{Leh1931b}.

Note that $(\1_r\oplus \1_r)(n_1,\ldots,n_r)=\tau(n_1^2)\cdots
\tau(n_r^2)$. It turns out that the lcm convolution preserves the
multiplicativity of functions, but ${\cal M}_r$ does not form a
group under the lcm convolution. The unity for the lcm convolution
is the function $\delta_r$ again. Here $({\cal A}_r,+,\oplus)$ is a
unital commutative ring having divisors of zero. The group of
invertible functions is $\widetilde{\cal A}_r = \{f\in {\cal A}_r:
(f*\1_r)(n_1,\ldots,n_r)\ne 0 \text{ for every } (n_1,\ldots,
n_r)\in \N^r\}$ and we deduce

\begin{proposition}
i)  Let $\widetilde{\cal M}_r= \{f\in {\cal M}_r:
(f*\1_r)(n_1,\ldots,n_r) \ne 0 \text{ for every } (n_1,\ldots
n_r)\in \N^r\}$. Then $\widetilde{\cal M}_r$ is a subgroup of
$\widetilde{\cal A}_r$ with respect to the lcm convolution.

ii) The inverse of the function $\1_r \in \widetilde{\cal A}_r$ is
$\mu^{\oplus}_r = \mu_r * 1/(\1_r*\1_r)$. The function
$\mu^{\oplus}_r$ is multiplicative and for every prime powers
$p^{\nu_1},\ldots,p^{\nu_r}$,
\begin{equation*}
\mu_r^{\oplus}(p^{\nu_1},\ldots,p^{\nu_r})=
\frac{(-1)^r}{\nu_1(\nu_1+1)\cdots \nu_r(\nu_r+1)}.
\end{equation*}
\end{proposition}

\begin{proof}
Here (ii) follows from $(\1_r*\1_r)(n_1,\ldots,n_r)=\tau(n_1)\cdots
\tau(n_r)$, already mentioned in Section \ref{subsection_Dirichlet_convolution}.
\end{proof}

\subsection{Binomial convolution}

We define the binomial convolution of the functions $f,g\in {\cal A}_r$ by
\begin{equation*}
(f\circ g)(n_1,\ldots,n_r)
\end{equation*}
\begin{equation*}
= \sum_{d_1\mid n_1,\ldots d_r\mid n_r}
\left( \prod_p \binom{\nu_p(n_1)}{\nu_p(d_1)} \cdots
\binom{\nu_p(n_r)}{\nu_p(d_r)} \right) f(d_1,\ldots,d_r)
g(n_1/d_1,\ldots,n_r/d_r),
\end{equation*}
where $\binom{a}{b}$ is the binomial coefficient. It is remarkable
that the binomial convolution preserves the complete multiplicativity of
arithmetical functions, which is not the case for the Dirichlet
convolution and other convolutions. Let $\xi_r$ be the firmly
multiplicative function given by
$\xi_r(n_1,\ldots,n_r)=\xi(n_1)\cdots \xi(n_r)$, that is
$\xi_r(n_1,\ldots,n_r)= \prod_p \left(\nu_p(n_1)!\cdots
\nu_p(n_r)!\right)$. Then for every $f,g\in {\cal A}_r$,
\begin{equation} \label{isom}
f\circ g =\xi_r \left(\frac{f}{\xi_r}*\frac{g}{\xi_r}\right),
\end{equation}
leading to the next result.

\begin{proposition} The algebras $({\cal A}_r,+,\circ,\C)$ and $({\cal A}_r,+,*,\C)$ are
isomorphic under the mapping $\ds f\mapsto\frac{f}{\xi_r}$.
\end{proposition}

Formula \eqref{isom} also shows that the binomial convolution
preserves the multiplicativity of functions. Furthermore, for any
fixed $r\in \N$ the structure $({\cal A}_r,+,\circ)$ is an integral
domain with unity $\delta_r$. The group of invertible functions is
again ${\cal A}_r^{(1)}$. If $f,g\in {\cal F}_r$, then with the notations of Proposition \ref{Prop_firmly_repr},
\begin{equation*}
(f\circ g)(n_1,\ldots,n_r)=(f_1\circ g_1)(n_1) \cdots (f_r \circ
g_r)(n_r)
\end{equation*}
and the inverse of $f$ is
\begin{equation*}
f^{-1\circ}(n_1,\ldots,n_r)=f_1^{-1\circ}(n_1) \cdots
f_r^{-1\circ}(n_r),
\end{equation*}
hence $f\circ g\in {\cal F}_r$ and $f^{-1\circ}\in {\cal F}_r$.
The inverse of the function $\1_r$ under the binomial convolution is
the function $\lambda_r$ given by $\lambda_r(n)=\lambda(n_1)\cdots
\lambda(n_r)$, i.e., $\lambda_r(n)=(-1)^{\Omega(n_1)+\ldots
+\Omega(n_r)}$. We deduce

\begin{proposition} One has the following subgroup relations:
\begin{equation*}
({\cal C}_r,\circ)\le ({\cal F}_r,\circ)\le ({\cal M}_r,\circ)\le
({\cal A}_r^{(1)},\circ).
\end{equation*}
\end{proposition}

In the case $r=1$ properties of this convolution were discussed in
the paper by {\sc L.~T\'oth, P.~Haukkanen} \cite{TotHau2009}. The proofs are similar in the
multivariable case.


\section{Generating series}

As generating series for multiplicative arithmetic functions of $r$
variables we present certain properties of the multiple Dirichlet
series, used earlier by several authors and the Bell series, which
constituted an important tool of {\sc R.~Vaidyanathaswamy}
\cite{Vai1931}.

\subsection{Dirichlet series}

The multiple Dirichlet series of a function $f\in {\cal A}_r$ is given by
\begin{equation*}
D(f;z_1,\ldots,z_r)= \sum_{n_1,\ldots,n_r=1}^{\infty}
\frac{f(n_1,\ldots,n_r)}{n_1^{z_1}\cdots n_r^{z_r}}.
\end{equation*}

Similar to the one variable case, if $D(f;z_1,\ldots,z_r)$ is
absolutely convergent in $(s_1,\ldots,s_r)\in \C^r$, then it is
absolutely convergent in every $(z_1,\ldots,z_r)\in \C^r$ with $\Re
z_j \ge \Re s_j$ ($1\le j\le r$).

\begin{proposition}  Let $f,g\in {\cal A}_r$. If $D(f;z_1,\ldots,z_r)$ and $D(g;z_1,\ldots,z_r)$ are
absolutely convergent, then
$D(f*g;z_1,\ldots,z_r)$ is also absolutely convergent and
\begin{equation*}
D(f*g;z_1,\ldots,z_r) = D(f;z_1,\ldots,z_r) D(g;z_1,\ldots,z_r).
\end{equation*}
\end{proposition}

Also, if $f\in {\cal A}_r^{(1)}$, then
\begin{equation*}
D(f^{-1_*};z_1,\ldots,z_r) = D(f;z_1,\ldots,z_r)^{-1},
\end{equation*}
formally or in the case of absolute convergence.

If $f\in {\cal M}_r$ is multiplicative, then its Dirichlet series
can be expanded into a (formal) Euler product, that is,
\begin{equation} \label{Euler_product}
D(f;z_1,\ldots,z_r)=  \prod_p \sum_{\nu_1,\ldots,\nu_r=0}^{\infty}
\frac{f(p^{\nu_1},\ldots, p^{\nu_r})}{p^{\nu_1z_1+\ldots +\nu_r
z_r}},
\end{equation}
the product being over the primes $p$. More exactly,

\begin{proposition}  Let $f\in {\cal M}_r$. For every $(z_1,\ldots,z_r)\in \C^r$ the series
$D(f;z_1,\ldots,z_r)$ is absolutely convergent if and only if
\begin{equation*}
\prod_p \sum_{\substack{\nu_1,\ldots,\nu_r=0\\ \nu_1+\ldots +\nu_r \ge 1}}^{\infty}
\frac{|f(p^{\nu_1},\ldots, p^{\nu_r})|}{p^{\nu_1 \Re z_1+\ldots +\nu_r
\Re z_r}} < \infty
\end{equation*}
and in this case the equality \eqref{Euler_product} holds.
\end{proposition}

Next we give the Dirichlet series representations of certain special
functions discussed above. These follow from the convolutional
identities given in Section \ref{subsection_examples}. For every
$g\in {\cal A}_1$ we have formally,
\begin{equation} \label{series_g}
\sum_{n_1,\ldots,n_r=1}^{\infty}
\frac{g(\gcd(n_1,\ldots,n_r))}{n_1^{z_1}\cdots n_r^{z_r}} =
\frac{\zeta(z_1)\cdots \zeta(z_r)}{\zeta(z_1+\ldots+z_r)}
\sum_{n=1}^{\infty} \frac{g(n)}{n^{z_1+\ldots+z_r}}.
\end{equation}

In particular, taking $g(n)=n$ \eqref{series_g} gives for $r\ge 2$
and $\Re z_1>1,\ldots,\Re z_r>1$,
\begin{equation*}
\sum_{n_1,\ldots,n_r=1}^{\infty}
\frac{\gcd(n_1,\ldots,n_r)}{n_1^{z_1}\cdots n_r^{z_r}} =
\frac{\zeta(z_1)\cdots \zeta(z_r)\zeta(z_1+\ldots+z_r-1)}
{\zeta(z_1+\ldots+z_r)}
\end{equation*}
and taking $g=\delta$ one obtains for $r\ge 2$ and $\Re
z_1>1,\ldots,\Re z_r>1$,
\begin{equation} \label{Dir_series_delta}
\sum_{\substack{n_1,\ldots,n_r=1\\ \gcd(n_1,\ldots,n_r)=1}}^{\infty}
\frac1{n_1^{z_1}\cdots n_r^{z_r}} = \frac{\zeta(z_1)\cdots
\zeta(z_r)}{\zeta(z_1+\ldots+z_r)},
\end{equation}
where the identity \eqref{Dir_series_delta} is the Dirichlet serries
of the characteristic function of the set of points in $\N^r$, which
are visible from the origin (cf. {\sc T.~M.~Apostol} \cite[Page
248,\ Ex. 15]{Apo1976}).

The Dirichlet series of the characteristic function $\varrho$
concerning $r$ pairwise relatively prime integers is
\begin{equation*}
\sum_{n_1,\ldots,n_r=1}^{\infty}
\frac{\varrho(n_1,\ldots,n_r)}{n_1^{z_1}\cdots n_r^{z_r}}
\end{equation*}
\begin{equation} \label{Dir_series_varrho}
= \zeta(z_1)\cdots \zeta(z_r) \prod_p \left( \prod_{j=1}^r
\left(1-\frac1{p^{z_j}}\right)+  \sum_{j=1}^r \frac1{p^{z_j}}
\prod_{\substack{k=1\\ k\ne j}}^r \left(1-\frac1{p^{z_k}}\right)
\right)
\end{equation}
\begin{equation} \label{Dir_series_varrho_2}
= \zeta(z_1)\cdots \zeta(z_r) \prod_p \left(1+ \sum_{j=2}^r (-1)^{j-1} (j-1) \sum_{1\le i_1<
\ldots < i_j\le r} \frac1{p^{z_{i_1}+\ldots +z_{i_j}}}\right),
\end{equation}
valid for $\Re z_1>1,\ldots,\Re z_r>1$. Here \eqref{Dir_series_varrho} follows at once from the definition
\eqref{def_pair_rel_prime} of the function $\varrho$. For \eqref{Dir_series_varrho_2}
see {\sc L.~T\'oth} \cite[Eq.\ (4.2)]{Tot2013IJNT}. Note that in the paper
\cite{Tot2013IJNT} certain other Dirichlet series representations
are also given, leading to generalizations of the Busche-Ramanujan
identities.

Concerning the Ramanujan sum $c_n(k)$ and the functions $s(m,n)$ and
$c(m,n)$, cf. Section \ref{subsection_examples}, Example
\ref{ex_s_c} we have for $\Re z>1,\Re w>1$,
\begin{equation*}
\sum_{k,n=1}^{\infty} \frac{c_n(k)}{k^z n^w} = \frac{\zeta(w)
\zeta(z+w-1)}{\zeta(z)},
\end{equation*}
\begin{equation} \label{Dir_s}
\sum_{m,n=1}^{\infty} \frac{s(m,n)}{m^z n^w} = \frac{\zeta^2(z)
\zeta^2(w)\zeta(z+w-1)}{\zeta(z+w)},
\end{equation}
\begin{equation} \label{Dir_c}
\sum_{m,n=1}^{\infty} \frac{c(m,n)}{m^z n^w} = \frac{\zeta^2(z)
\zeta^2(w)\zeta(z+w-1)}{\zeta^2(z+w)}.
\end{equation}

The formulas \eqref{Dir_s} and \eqref{Dir_c} were derived by {\sc
W.~G.~Nowak, L.~T\'oth} \cite{NowTot2013}. Also, for $\Re z>2,\Re
w>2$,
\begin{equation*}
\sum_{m,n=1}^{\infty} \frac{\lcm(m,n)}{m^zn^w} =
\frac{\zeta(z-1)\zeta(w-1)\zeta(z+w-1)}{\zeta(z+w-2)}.
\end{equation*}

As a generalization of \eqref{Dir_s}, for $\Re z_1>1,\ldots,\Re
z_r>1$,
\begin{equation*}
\sum_{n_1,\ldots,n_r=1}^{\infty}
\frac{\sigma(n_1,\ldots,n_r)}{n_1^{z_1}\cdots n_r^{z_r}} =
\frac{\zeta^2(z_1)\cdots \zeta^2(z_r)\zeta(z_1+\ldots+z_r-1)}
{\zeta(z_1+\ldots+z_r)}.
\end{equation*}


\subsection{Bell series}

If $f$ is a multiplicative function of $r$ variables, then its (formal) Bell series to the base $p$ ($p$
prime) is defined by
\begin{equation*}
f_{(p)}(x_1,\ldots,x_r)= \sum_{e_1,\ldots,e_r=0}^{\infty}
f(p^{e_1},\ldots,p^{e_r}) x_1^{e_1}\cdots x_r^{e_r},
\end{equation*}
where the constant term is $1$. The main property is the following:
for every $f,g\in {\cal M}_r$,
\begin{equation*}
(f*g)_{(p)}(x_1,\ldots,x_r)=f_{(p)}(x_1,\ldots,x_r)
g_{(p)}(x_1,\ldots,x_r).
\end{equation*}

The connection of Bell series to Dirichlet series and Euler products
is given by
\begin{equation} \label{Connect_Bell_Euler}
D(f;z_1,\ldots,z_r)= \prod_p f_{(p)}(p^{-z_1},\ldots,p^{-z_r}),
\end{equation}
valid for every $f\in {\cal M}_r$. For example, the Bell series of
the gcd function $f(n_1,\ldots,n_r)=\gcd(n_1,\ldots,n_r)$ is
\begin{equation*}
f_{(p)}(x_1,\ldots,x_r)= \frac{1-x_1\cdots x_r}{(1-x_1)\cdots (1-x_r)(1-px_1\cdots x_r)}.
\end{equation*}

The Bell series of other multiplicative functions, in particular of
$c(m,n)$, $s(m,n)$, $\sigma(n_1,\ldots,n_r)$ and $c_n(k)$ can be
given from their Dirichlet series representations and using the
relation \eqref{Connect_Bell_Euler}.

Note that in the one variable case the Bell series to a fixed prime of the
unitary convolution of two multiplicative
functions is the sum of the Bell series of the functions, that is
\begin{equation*}
(f\times g)_{(p)}(x_1)=f_{(p)}(x_1)+ g_{(p)}(x_1),
\end{equation*}
see {\sc R.~Vaidyanathaswamy} \cite[Th.\ XII]{Vai1931}. This is not valid in the case of
$r$ variables with $r>1$.


\section{Convolutes of arithmetic functions of several variables}

Let $f\in {\cal A}_r$. By choosing $n_1=\ldots =n_r=n$ we obtain the
function of a single variable $n\mapsto \overline{f}(n)=
f(n,\ldots,n)$. If $f\in {\cal M}_r$, then $\overline{f}\in {\cal
M}_1$, as already mentioned. Less trivial ways to retrieve from $f$
functions of a single variable is to consider for $r>1$,
\begin{equation} \label{opdir}
\opdir(f)(n)= \sum_{d_1\cdots d_r = n} f(d_1,\ldots, d_r),
\end{equation}
\begin{equation} \label{opunit}
\opunit(f)(n)= \sum_{\substack{d_1\cdots d_r = n\\
\gcd(d_i,d_j)=1, i\ne j}} f(d_1,\ldots, d_r),
\end{equation}
\begin{equation} \label{opgcd}
\opgcd(f)(n)= \sum_{\substack{d_1\cdots d_r = n\\
\gcd(d_1,\ldots,d_r)=1}} f(d_1,\ldots, d_r),
\end{equation}
\begin{equation} \label{oplcm}
\oplcm(f)(n)= \sum_{\lcm(d_1,\ldots, d_r) = n} f(d_1,\ldots, d_r),
\end{equation}
\begin{equation} \label{opbinom}
\opbinom(f)(n)= \sum_{d_1\cdots d_r=n} \left(\prod_p
\binom{\nu_p(n)}{\nu_p(d_1),\ldots,\nu_p(d_r)}\right) f(d_1,\ldots,
d_r),
\end{equation}
where the sums are over all ordered $r$-tuples $(d_1,\ldots,d_r)\in
\N^r$ with the given additional conditions, the last one involving
multinomial coefficients. For \eqref{opunit} the condition is that
$d_1\cdots d_r = n$ and $d_1,\ldots,d_r$ are pairwise relatively
prime. Note that \eqref{opunit} and \eqref{opgcd} are the same for
$r=2$, but they differ in the case $r>2$.

Assume that there exist functions $g_1,\ldots,g_r\in {\cal A}_1$
(each of a single variable) such that $f(n_1,\ldots,n_r)=
g_1(n_1)\cdots g_r(n_r)$ (in particular this holds if $f\in {\cal
F}_r$ by Proposition \ref{Prop_firmly_repr}). Then \eqref{opdir},
\eqref{opunit}, \eqref{oplcm} and \eqref{opbinom} reduce to the Dirichlet
convolution, unitary convolution, lcm convolution and binomial convolution, respectively,
of the functions $g_1,\ldots, g_r$. For $r=2$ we have the
corresponding convolutions of two given functions of a single
variable.

Note the following special case of \eqref{oplcm}:
\begin{equation*}
\sum_{\lcm(d_1,\ldots,d_r)=n} \phi(d_1)\cdots \phi(d_r)= \phi_r(n)
\qquad (n\in \N),
\end{equation*}
due to {\sc R.~D.~von Sterneck} \cite{Ste1894}.

We remark that \eqref{opdir}, with other notation, appears in
\cite[p.\ 591-592]{Vai1931}, where $\opdir(f)$ is called the
'convolute' of $f$ (obtained by the convolution of the arguments).
We will call $\opdir(f)$, $\opunit(f)$, $\opgcd(f)$, $\oplcm(f)$ and
$\opbinom(f)$ the {\sl Dirichlet convolute}, {\sl unitary
convolute}, {\sl gcd convolute}, {\sl lcm convolute} and {\sl
binomial convolute}, respectively of the function $f$.

Some special cases of convolutes of functions which are not the
product of functions of a single variable are the following. Special
Dirichlet convolutes are
\begin{equation} \label{def_g_r}
g_r(n)= \sum_{d_1\cdots d_r=n} \gcd(d_1,\ldots,d_r),
\end{equation}
\begin{equation} \label{def_ell_r}
\ell_r(n)= \sum_{d_1\cdots d_r=n} \lcm(d_1,\ldots,d_r).
\end{equation}
\begin{equation} \label{N}
N(n)= \sum_{d\mid n} \phi(\gcd(d,n/d)).
\end{equation}

For $r=2$ \eqref{def_g_r} and \eqref{def_ell_r} are sequences
A055155 and A057670, respectively in \cite{OEIS}. The function
$N(n)$ given by \eqref{N} represents the number of parabolic
vertices of $\gamma_0(n)$ (sequence A001616 in \cite{OEIS}), cf.
{\sc S.~Finch} \cite{Fin2005}.

In the case $r=2$ the lcm convolute of the gcd function
\begin{equation} \label{def_c}
c(n)= \sum_{\lcm(d,e)=n} \gcd(d,e),
\end{equation}
represents the number of cyclic subgroups of the group $\Z_n \times
\Z_n$, as shown by {\sc A.~Pakapongpun, T.~Ward} \cite[Ex.\
2]{PakWar2009}. That is, $c(n)=c(n,n)$ for every $n\in \N$, with the
notation of Section \ref{subsection_examples}, Example \ref{ex_s_c}
(it is sequence A060648 in \cite{OEIS}).

The Dirichlet, unitary and lcm convolutes of the Ramanujan sums are
\begin{equation} \label{def_a}
a(n)= \sum_{d\mid n} c_d(n/d),
\end{equation}
\begin{equation} \label{def_b}
b(n)= \sum_{d\mid\mid n} c_d(n/d),
\end{equation}
\begin{equation} \label{def_h}
h(n)=\sum_{\lcm(d,e)=n} c_d(e).
\end{equation}

All the functions $g_r, \ell_r,N,c,a,b,h$ defined above are
multiplicative, as functions of a single variable. See Corollary
\ref{cor_conv_multipl}.

\subsection{General results}

The Dirichlet, unitary, gcd, lcm and binomial convolutions preserve the
multiplicativity of functions of a single variable, cf. Section
\ref{sect_convo}. As a generalization of this property, we prove the next result.

\begin{proposition} \label{op_multipl} Let $f\in {\cal M}_r$ be an arbitrary multiplicative function.
Then all the functions $\opdir(f)$, $\opunit(f)$, $\opgcd(f)$,
$\oplcm(f)$ and $\opbinom(f)$ are multiplicative.
\end{proposition}

Note that for the Dirichlet covolute this property was pointed out
by {\sc R. Vaidyanathaswamy} in \cite[p.\ 591-592]{Vai1931}.

\begin{proof} By the definitions. Let $n,m\in \N$ such that $\gcd(n,m)=1$.
If $d_1\cdots d_r=nm$, then there exist unique integers
$a_1,b_1,\ldots, a_r,b_r\in \N$ such that $a_1,\ldots,a_r\mid n$,
$b_1,\ldots,b_r\mid m$ and $d_1=a_1b_1, \ldots, d_r=a_rb_r$. Here
$\gcd(a_1\cdots a_r,b_1\cdots b_r)=1$. Using the multiplicativity of
$f$ we obtain
\begin{equation*}
\opdir(f)(nm)= \sum_{\substack{a_1\cdots a_r = n\\ b_1\cdots b_r=m}}
f(a_1b_1,\ldots, a_rb_r)
\end{equation*}
\begin{equation*}
= \sum_{a_1\cdots a_r = n} f(a_1,\ldots, a_r) \sum_{b_1\cdots b_r =
m} f(b_1,\ldots, b_r)
\end{equation*}
\begin{equation*}
= \opdir(f)(n) \opdir(f)(m),
\end{equation*}
showing the multiplicativity of $\opdir(f)$. The proof in the case
of the other functions is similar. For the function $\opgcd(f)$,
\begin{equation*}
\opgcd(f)(nm)= \sum_{\substack{a_1\cdots a_r = n\\ b_1\cdots b_r=m\\
\gcd(a_1b_1,\ldots,a_rb_r)=1}} f(a_1b_1,\ldots, a_rb_r),
\end{equation*}
where $1= \gcd(a_1b_1,\ldots,a_rb_r)= \gcd(a_1,\ldots,a_r)
\gcd(b_1,\ldots,b_r)$, since the gcd function in $r$ variables is multiplicative. Hence
\begin{equation*}
\opgcd(f)(nm)= \sum_{\substack{a_1\cdots a_r = n\\
\gcd(a_1,\ldots, a_r)=1}} f(a_1,\ldots, a_r) \sum_{\substack{b_1\cdots b_r = m\\
\gcd(b_1,\ldots, b_r)=1}} f(b_1,\ldots, b_r)
\end{equation*}
\begin{equation*}
= \opgcd(f)(n) \opgcd(f)(m).
\end{equation*}

In the case of the function $\oplcm(f)$ use that the lcm function in $r$ variables is
multiplicative, whence $nm = \lcm(a_1b_1,\ldots,a_rb_r)=
\lcm(a_1,\ldots,a_r) \lcm(b_1,\ldots,b_r)$ and it follows that
$\lcm(a_1,\ldots,a_r)=n$, $\lcm(b_1,\ldots,b_r)=m$.
\end{proof}

\begin{remark} {\rm Alternative proofs for the multiplicativity of $\opunit(f)$,
$\opgcd(f)$ and $\oplcm(f)$ can be given as follows. In the first
two cases the property can be reduced to that of $\opdir(f)$. Let
\begin{equation*}
f^{\flat}(n_1,\ldots,n_r)= \begin{cases} f(n_1,\ldots,n_r), & \text{
if } \gcd(n_1,\ldots,n_r)=1,\\ 0, & \text{otherwise}.
\end{cases}
\end{equation*}

Then $\opgcd(f)= \opdir(f^{\flat})$. If $f$ is multiplicative, then
$f^{\flat}$ is also multiplicative and the multiplicativity of
$\opgcd(f)$ follows by the same property of the Dirichlet convolute.
Similar for $\opunit(f)$. Furthermore,
\begin{equation} \label{form_oplcm}
\oplcm(f) = (\overline{f*\1_r})*\mu.
\end{equation}
Indeed,  we have similar to the proof of \eqref{lcm_convo_connection} given above,
\begin{equation*}
\sum_{d\mid n} \oplcm(f)(d)= \sum_{\lcm(d_1,\ldots,d_r)\mid d}
f(d_1,\ldots,d_r)= \sum_{d_1\mid n,\ldots, d_r\mid n}
f(d_1,\ldots,d_r)
\end{equation*}
\begin{equation*}
= (f * \1_r)(n,\ldots,n)=(\overline{f*\1_r})(n).
\end{equation*}

If $f$ is multiplicative, so is $f* \1_r$ (as a function of $r$
variables). Therefore, $\overline{f * \1_r}$ is multiplicative (as a
function of a single variable) and deduce by \eqref{form_oplcm} that
$\oplcm(f)$ is multiplicative.}
\end{remark}

\begin{remark} \label{opbinom_compl_mult} Note that if $f$ is completely multiplicative, then $\opbinom(f)$ is also
completely multiplicative.
\end{remark}

\begin{corollary} \label{cor_conv_multipl} The functions $g_r, \ell_r,N,c,a,b,h$ defined by
\eqref{def_g_r}-\eqref{def_h} are multiplicative.
\end{corollary}

\begin{proposition} \label{Prop_number_terms} The convolutes of the function $f=\1_r$, that
is the number of terms of the sums defining the convolutes are the following multiplicative functions:

i) $\opdir(\1_r)=\tau_r$, the Piltz divisor function of order $r$ given by $\tau_r(n)=\prod_p
\binom{\nu_p(n)+r-1}{r-1}$,

ii) $\opunit(\1_r)=H_r$ given by $H_r(n)=r^{\omega(n)}$,

iii) $\opgcd(\1_r)=N_r$, where $N_r(n)=\sum_{a^rb=n} \mu(a)\tau_r(b) =
\prod_p \left( \binom{\nu_p(n)+r-1}{r-1}-
\binom{\nu_p(n)-1}{r-1} \right)$ with $\binom{\nu-1}{r-1}=0$
for $\nu<r$,

iv) $\oplcm(\1_r)=M_r$, where $M_r(n)= \sum_{ab=n} \mu(a)\tau(b)^r
= \prod_p \left((\nu_p(n)+1)^r-\nu_p(n)^r \right)$,

v) $\opbinom(\1_r)=Q_r$, where $Q_r(n)=r^{\Omega(n)}$.
\end{proposition}

\begin{proof} i) and ii) are immediate from the definitions.

iii) By the property of the M\"obius function,
\begin{equation*}
\opgcd(\1_r)(n) = \sum_{d_1\cdots d_r=n} \ \sum_{a\mid
\gcd(d_1,\ldots,d_r)} \mu(a)= \sum_{a^kb_1\cdots b_r=n} \mu(a) =
\sum_{a^rb=n} \mu(a)\tau_r(b).
\end{equation*}

iv) Follows from \eqref{form_oplcm} in the case $f=\1_r$.

v) By Remark \ref{opbinom_compl_mult}. See also {\sc L.~T\'oth, P.~Haukkanen} \cite[Cor.\ 3.2]{TotHau2009}.
\end{proof}

Here i) and iv) of Proposition \ref{Prop_number_terms} can be
generalized as follows.

\begin{proposition} \label{prop_g_gcd} Assume that there is a function $g\in {\cal A}_1$  (of a single variable)
such that $f(n_1,\ldots,n_r)=g(\gcd(n_1,\ldots,n_r))$ for every
$n_1,\ldots,n_r\in \N$. Then for every $n\in \N$,
\begin{equation} \label{dir_g_gcd}
\opdir(f)(n)= \sum_{d_1\cdots d_r=n} g(\gcd(d_1,\ldots,d_r)) =
\sum_{a^rb=n} (\mu*g)(a)\tau_r(b),
\end{equation}
\begin{equation} \label{lcm_g_gcd}
\oplcm(f)(n) = \sum_{\lcm(d_1,\ldots,d_r)=n} g(\gcd(d_1,\ldots,d_r))
= (g*\mu* \mu *\tau^r)(n).
\end{equation}
\end{proposition}

\begin{proof} The identity \eqref{dir_g_gcd} is given by {\sc E.~Kr\"atzel, W.~G.~Nowak, L.~T\'oth}
\cite[Prop.\ 5.1]{KraNowTot2012}. We recall its proof, which is
simple and similar to that of iii) of Proposition
\ref{Prop_number_terms}: for $f(n_1,\ldots,n_r) =
g(\gcd(n_1,\ldots,n_r))$,
\begin{equation*}
\opdir(f)(n) = \sum_{d_1\cdots d_r=n} \ \sum_{a\mid
\gcd(d_1,\ldots,d_r)} (\mu*g)(a)
\end{equation*}
\begin{equation*}
= \sum_{a^rb_1\cdots b_r=n}
(\mu*g)(a) = \sum_{a^rb=n} (\mu*g)(a)\tau_r(b).
\end{equation*}

Now for \eqref{lcm_g_gcd},
\begin{equation*}
\oplcm(f)(n) = \sum_{\lcm(d_1,\ldots, d_r)=n} \ \sum_{a\mid
\gcd(d_1,\ldots,d_r)} (\mu*g)(a)
\end{equation*}
\begin{equation*}
= \sum_{a\mid n} (\mu*g)(a) \sum_{\lcm(b_1,\ldots, b_r)=n/a} 1 =
\sum_{a\mid n} (\mu*g)(a)(\mu*\tau^r)(n/a)
\end{equation*}
\begin{equation*}
= (g*\mu* \mu *\tau^r)(n),
\end{equation*}
using iv) of Proposition \ref{Prop_number_terms}.
\end{proof}

\begin{proposition} \label{morfism} For every $f,g\in {\cal A}_r$,
\begin{equation*}
\opdir(f*g) = \opdir(f)*\opdir(g),
\end{equation*}
\begin{equation*}
\opunit(f\times g) = \opunit(f) \times \opunit(g),
\end{equation*}
\begin{equation*}
\opgcd(f\odot g) = \opgcd(f) \times \opgcd(g),
\end{equation*}
\begin{equation*}
\oplcm(f\oplus g) = \oplcm(f) \oplus \oplcm(g),
\end{equation*}
\begin{equation*}
\opbinom(f\circ g) = \opbinom(f) \circ \opbinom(g).
\end{equation*}
\end{proposition}

\begin{proof} For the Dirichlet convolute $\opdir$,
\begin{equation*}
\opdir(f*g)(n) = \sum_{d_1\cdots d_r=n} (f*g)(d_1,\ldots,d_r)
\end{equation*}
\begin{equation*}
= \sum_{d_1\cdots d_r=n} \sum_{a_1b_1=d_1,\ldots, a_rb_r=d_r}
f(a_1,\ldots,a_r) g(b_1,\ldots,b_r)
\end{equation*}
\begin{equation*}
= \sum_{a_1b_1\cdots a_rb_r=n} f(a_1,\ldots,a_r) g(b_1,\ldots,b_r)
\end{equation*}
\begin{equation*}
= \sum_{xy=n} \sum_{a_1\cdots a_r=x} f(a_1,\ldots,a_r)
\sum_{b_1\cdots b_r=y} g(b_1,\ldots,b_r)
\end{equation*}
\begin{equation*}
= \sum_{xy=n} \opdir(f)(x) \opdir(g)(y) =(\opdir(f)*\opdir(g))(n),
\end{equation*}
and similar for the other ones.
\end{proof}

\begin{proposition} \label{algebra_morfism} Let $r\ge 2$. The following maps are surjective algebra homomorphisms:
$\opdir \colon ({\cal A}_r,+,\cdot,*) \to ({\cal A}_1,+,\cdot,*)$,
$\opunit \colon ({\cal A}_r,+,\cdot,\times) \to ({\cal
A}_1,+,\cdot,\times)$, \\ $\opgcd \colon ({\cal A}_r,+,\cdot,\odot) \to
({\cal A}_1,+,\cdot,\times)$, $\oplcm \colon ({\cal
A}_r,+,\cdot,\oplus) \to ({\cal A}_1,+,\cdot,\oplus)$, \\
$\opbinom \colon ({\cal A}_r,+,\cdot,\circ) \to ({\cal A}_1,+,\cdot,\circ)$.
\end{proposition}

\begin{proof} Use Proposition \ref{morfism}. For the surjectivity: for a given
$f\in {\cal A}_1$ consider $F\in {\cal A}_r$ defined by $F(n,1,\ldots,1)=f(n)$ for every $n\in \N$ and
$F(n_1,\ldots,n_r)=0$ otherwise, i.e., for every $n_1,\ldots,n_r\in
\N$ with $n_2\cdots n_r>1$. Then $\opdir(F)=\opunit(F)=\opgcd(F)=\oplcm(F)=\opbinom(F)=f$.
\end{proof}

\begin{corollary} Let $r\ge 2$.

i) The maps
$\opdir \colon ({\cal M}_r,*) \to ({\cal M}_1,*)$,
$\opunit \colon ({\cal M}_r,\times) \to ({\cal M}_1,\times)$, \\
$\opgcd \colon ({\cal M}_r,\odot) \to ({\cal M}_1,\times)$ and
$\opbinom \colon ({\cal M}_r,\circ) \to ({\cal M}_1,\circ)$ are surjective group homomorphisms.

ii) The maps
$\opdir \colon ({\cal F}_r,*) \to ({\cal M}_1,*)$,
$\opunit \colon ({\cal F}_r,\times) \to ({\cal M}_1,\times)$ and \\
$\opbinom \colon ({\cal F}_r,\circ) \to ({\cal M}_1,\circ)$
are surjective group homomorphisms.

iii) The map $\opbinom \colon ({\cal C}_r,\circ) \to ({\cal C}_1,\circ)$ is a surjective group homomorphism.
\end{corollary}

\begin{proof} Follows from Propositions \ref{op_multipl}, \ref{morfism} and from the fact that
for every (completely) multiplicative $f$ the function $F$ constructed in
the proof of Proposition \ref{algebra_morfism} is (completely) multiplicative. For iii) use also Remark \ref{opbinom_compl_mult}.
\end{proof}

\subsection{Special cases}

We present identities for the convolutes of some special functions.
For Dirichlet convolutes we have the next result.

\begin{corollary} For every $k\in \C$,
\begin{equation*}
\sum_{d_1\cdots d_r=n} (\gcd(d_1,\ldots,d_r))^k= \sum_{a^rb=n}
\phi_k(a)\tau_r(b).
\end{equation*}
\begin{equation} \label{id_cnvol_sigma}
\sum_{d_1\cdots d_r=n} \sigma_k(\gcd(d_1,\ldots,d_r))= \sum_{a^rb=n}
a^k \tau_r(b).
\end{equation}
\end{corollary}

\begin{proof} Follows from the first identity of Proposition
\ref{prop_g_gcd}.
\end{proof}

The function \eqref{id_cnvol_sigma} is for $r=2$ and $k=0$ the
sequence A124315 in \cite{OEIS}, and for $r=2$, $k=1$ it is sequence
A124316 in \cite{OEIS}. See also \cite[Sect.\ 5]{KraNowTot2012}.

Special cases of lcm convolutes which do not seem to be known are
the following. Let $\beta=\id* \lambda$ be the alternating
sum-of-divisors function, see {\sc L.~T\'oth} \cite{Tot2013Sap}.

\begin{corollary} For every $n\in \N$,
\begin{equation*}
\sum_{\lcm(d_1,\ldots,d_r)=n} \gcd(d_1,\ldots,d_r) =
(\phi*M_r)(n),
\end{equation*}
where the function $M_r$ is given in Proposition \ref{Prop_number_terms},
\begin{equation*}
\sum_{\lcm(d_1,\ldots,d_r)=n} \tau(\gcd(d_1,\ldots,d_r)) =
\tau(n)^r,
\end{equation*}
\begin{equation*}
\sum_{\lcm(d,e)=n} \phi(\gcd(d,e)) = \psi(n),
\end{equation*}
\begin{equation} \label{form_number_subgroups}
\sum_{\lcm(d,e)=n} \sigma(\gcd(d,e)) = (\tau* \psi)(n)= (\phi *
\tau^2)(n),
\end{equation}
\begin{equation*}
\sum_{\lcm(d,e)=n} \beta(\gcd(d,e)) = \sigma(n),
\end{equation*}
\begin{equation*}
\sum_{\lcm(d,e)=n} \mu(\gcd(d,e)) = \mu^2(n),
\end{equation*}
\begin{equation*}
\sum_{\lcm(d,e)=n} \lambda(\gcd(d,e)) = 1.
\end{equation*}
\end{corollary}

\begin{proof} Follow from the second identity of Proposition
\ref{prop_g_gcd}.
\end{proof}

Here \eqref{form_number_subgroups} is $s(n,n)$, representing the
number of all subgroups of the group $\Z_n \times \Z_n$ (sequence
A060724 in \cite{OEIS}).

For the convolutes of the Ramanujan sum we have
\begin{proposition}
\begin{equation} \label{Raman_Dir_convolute}
\sum_{de=n} c_d(e) = \begin{cases} \sqrt{n}, & n \ \text{ perfect square},\\
0, & \text{otherwise},
\end{cases}
\end{equation}
\begin{equation} \label{Raman_unit_convolute}
\sum_{\substack{de=n \\ \gcd(d,e)=1}} c_d(e) = \begin{cases} 1, & n \ \text{ squarefull},\\
0, & \text{otherwise},
\end{cases}
\end{equation}
\begin{equation} \label{Raman_lcm_convolute}
\sum_{\lcm(d,e)=n} c_d(e) = \phi(n) \quad (n\in \N).
\end{equation}
\end{proposition}

\begin{proof} According to Corollary \ref{cor_conv_multipl} all these functions are
multiplicative and it is enough to compute their values for prime
powers. Alternatively, the formula \eqref{formula_Ramanujan} can be
used.
\end{proof}

Formulas \eqref{Raman_Dir_convolute} and
\eqref{Raman_unit_convolute} are well known, see, e.g., {\sc
P.~J.~McCarthy} \cite[p.\ 191--192]{McC1986}, while
\eqref{Raman_lcm_convolute} seems to be new.

\section{Asymptotic properties}

We discuss certain results concerning the mean values of arithmetic
functions of $r$ variables and the asymptotic densities of some sets in
$\N^r$. We also present asymptotic formulas for special
multiplicative functions given in the previous sections.

\subsection{Mean values}

Let $f\in {\cal A}_r$. The mean value of $f$ is
\begin{equation*}
M(f)= \lim_{x_1,\ldots, x_r\to \infty} \frac1{x_1\cdots x_r}
\sum_{n_1\le x_1,\ldots, n_r\le x_r} f(n_1,\ldots,n_r),
\end{equation*}
where $x_1,\ldots,x_r$ tend to infinity independently, provided that
this limit exists. As a generalization of Wintner's theorem (valid
for the case $r=1$), {\sc N.~Ushiroya} \cite[Th.\ 1]{Ush2012} proved
the next result.

\begin{proposition} \label{Prop_Dir_Wintner} If $f\in {\cal A}_r$ ($r\ge 1$) and
\begin{equation*}
\sum_{n_1,\ldots,n_r=1}^{\infty} \frac{|(\mu_r*f)(n_1,\ldots,n_r)|}{n_1\cdots n_r} < \infty,
\end{equation*}
then the mean value $M(f)$ exists and
\begin{equation*}
M(f)= \sum_{n_1,\ldots,n_r=1}^{\infty} \frac{(\mu_r*f)(n_1,\ldots,n_r)}{n_1\cdots n_r}.
\end{equation*}
\end{proposition}

For multiplicative functions we have the following result due to {\sc N.~Ushiroya} \cite[Th.\ 4]{Ush2012} in a
slightly different form.

\begin{proposition} Let $f\in {\cal M}_r$ ($r\ge 1$). Assume that
\begin{equation*}
\sum_p \sum_{\substack{\nu_1,\ldots,\nu_r=0\\ \nu_1+\ldots +\nu_r \ge 1}}^{\infty}
\frac{|(\mu_r*f)(p^{\nu_1},\ldots,p^{\nu_r})|}{p^{\nu_1+\ldots +\nu_r}} < \infty.
\end{equation*}

Then the mean value $M(f)$ exists and
\begin{equation*}
M(f)= \prod_p \left(1-\frac1{p}\right)^r \sum_{\nu_1,\ldots,\nu_r=0}^{\infty}
\frac{f(p^{\nu_1},\ldots,p^{\nu_r}
)}{p^{\nu_1+\ldots +\nu_r}}.
\end{equation*}
\end{proposition}

\begin{corollary} {\rm ({\sc N.~Ushiroya} \cite[Th.\ 7]{Ush2012})} \label{cor_mean_value}
Let $g\in {\cal M}_1$ be a multiplicative function and denote by
$a_g$ the absolute convergence abscissa of the Dirichlet series
$D(g;z)$. Then for every $r>1$, $r>a_g$ the main value of the
function $(n_1,\ldots,n_r) \mapsto g(\gcd(n_1,\ldots,n_r))$ exists
and
\begin{equation*}
M(f)= \frac1{\zeta(r)} \sum_{n=1}^{\infty} \frac{g(n)}{n^r}.
\end{equation*}
\end{corollary}

\begin{proof} Follows from Proposition \ref{Prop_Dir_Wintner} and the identity \eqref{series_g}.
\end{proof}

For example, the mean value of the function $(n_1,\ldots,n_r) \mapsto \gcd(n_1,\ldots,n_r)$ is
$\zeta(r-1)/\zeta(r)$ ($r\ge 3$), the mean value of the function $(n_1,\ldots,n_r) \mapsto \phi(\gcd(n_1,\ldots,n_r))$ is
$\zeta(r-1)/\zeta^2(r)$ ($r\ge 3$).

The analog of Proposition \ref{Prop_Dir_Wintner} for the unitary
convolution is the next result (see {\sc W. Narkiewicz}
\cite{Nar1963} in the case $r=1$).

\begin{proposition} \label{Prop_unitary_Wintner} If $f\in {\cal A}_r$ ($r\ge 1$) and
\begin{equation*}
\sum_{n_1,\ldots,n_r=1}^{\infty} \frac{|(\mu^{\times}_r\times f)(n_1,\ldots,n_r)|}{n_1\cdots n_r} < \infty,
\end{equation*}
then the mean value $M(f)$ exists and
\begin{equation*}
M(f)= \sum_{n_1,\ldots,n_r=1}^{\infty} \frac{(\mu^{\times}_r\times f)(n_1,\ldots,n_r)\phi(n_1)\cdots \phi(n_r)}{n^2_1\cdots n^2_r}.
\end{equation*}
\end{proposition}

For further results on the mean values of multiplicative arithmetic
functions of several variables and generalizations to the several
variables case of results of {\sc G.~Hal\'asz} \cite{Hal1969} we
refer to the papers of {\sc O.~Casas} \cite{Cas2006}, {\sc
H.~Delange} \cite{Del1969, Del1970}, {\sc E.~Heppner} \cite{Hep1980,
Hep1981} and {\sc K-H.~Indlekofer} \cite{Ind1972}. See also {\sc
E.~Alkan, A.~Zaharescu, M.~Zaki} \cite{AlkZahZak2007}.

\subsection{Asymptotic densities}

Let $S\subset \N^r$. The set $S$ possesses the asymptotic density
$d_S$ if the characteristic function $\chi_S$ of $S$ has the mean
value $M(\chi_S)=d_S$. In what follows we consider the densities of
certain special sets.

Let $g\in {\cal M}_1$ be a multiplicative function such that
$g(n)\in \{0,1\}$ for every $n\in \N$. Let
$S_g=\{(n_1,\ldots,n_r)\in \N^r: g(\gcd(n_1,\ldots,n_r))=1\}$. It
follows from Corollary \ref{cor_mean_value} that for $r\ge 2$ the
set $S_g$ has the asymptotic density given by
\begin{equation*}
d_{S_g}= \frac1{\zeta(r)} \sum_{n=1}^{\infty} \frac{g(n)}{n^r}.
\end{equation*}

In particular, for $r\ge 2$ the set of points $(n_1,\ldots,n_r)\in
\N^r$ which are visible from the origin, i.e., such that
$\gcd(n_1,\ldots,n_r)=1$ holds has the density $1/\zeta(r)$ (the
case $g=\delta$). Another example: the set of points
$(n_1,\ldots,n_r)\in \N^r$ such that $\gcd(n_1,\ldots,n_r)$ is
squarefree has the density $1/\zeta(2r)$ (the case $g=\mu^2$). These
results are well known.

Now consider the set of points $(n_1,\ldots,n_r)\in \N^r$ such that $n_1,\ldots,n_r$ are pairwise relatively prime. The next results was
first proved by {\sc L.~T\'oth} \cite{Tot2002}, giving also an asymptotic formula for $\sum_{n_1,\ldots,n_r\le x} \varrho(n_1,\ldots,n_r)$
and by {\sc J.-Y.~Cai, E.~Bach} \cite[Th.\ 3.3]{CaiBac2003}. Here we give a simple different proof.

\begin{proposition} \label{Prop_pairw_rel_prime} Let $r\ge 2$. The asymptotic density of the set of points in $\N^r$ with pairwise
relatively prime coordinates
is
\begin{equation} \label{A_r}
A_r=\prod_p \left(1-\frac1{p}\right)^{r-1}\left(1+\frac{r-1}{p}\right).
\end{equation}
\end{proposition}

\begin{proof} Apply Proposition \ref{Prop_Dir_Wintner} for the function $f=\varrho$ defined by
\eqref{def_pair_rel_prime}. Then according to the Dirichlet series representation \eqref{Dir_series_varrho}, the density is
\begin{equation*}
\sum_{n_1,\ldots,n_r=1}^{\infty}
\frac{(\mu_r*\varrho)(n_1,\ldots,n_r)}{n_1\cdots n_r} =  \prod_p \left(\left(1-\frac1{p}\right)^r + \frac{r}{p}\left(1-\frac1{p}\right)^{r-1}\right),
\end{equation*}
which equals $A_r$, given by \eqref{A_r}.
\end{proof}

See the quite recent paper by {\sc J.~Hu} \cite{Hu2013} for a
generalization of Proposition \ref{Prop_pairw_rel_prime}. See {\sc
J.~L.~Fern\'andez, P.~Fern\'andez}
\cite{FF2013PR,FF2013NT,FF2013NToct,FF2013NToct2} for various statistical regularity properties
concerning mutually relatively prime and pairwise relatively prime integers.

Unitary analogs of the problems of above are the following.

\begin{proposition} \label{Prop_unit_rel_prime} Let $r\ge 2$. The asymptotic density of the set of points $(n_1,\ldots,n_r)\in \N^r$
such that $\gcud(n_1,\ldots,n_r)=1$ is
\begin{equation*}
\prod_p \left(1-\frac{(p-1)^r}{p^r(p^r-1)}\right).
\end{equation*}
\end{proposition}

\begin{proof} The characteristic function of this set is given by
\begin{equation*}
\delta(\gcud(n_1,\ldots,n_r)) = \sum_{d\mid\mid
\gcud(n_1,\ldots,n_r)} \mu^{\times}(d) = \sum_{d_1\mid\mid n_r,
\ldots, d_r\mid\mid n_r} G(d_1,\ldots,d_r),
\end{equation*}
where
\begin{equation*}
G(n_1,\ldots,n_r) = \begin{cases} \mu^{\times}(n), & \text{if
$n_1=\ldots=n_r=n$}, \\ 0, & \text{otherwise}. \end{cases}
\end{equation*}

We deduce from Proposition \ref{Prop_unitary_Wintner} that the
density in question is
\begin{equation*}
\sum_{n_1,\ldots,n_r=1}^{\infty}
\frac{G(n_1,\ldots,n_r)\phi(n_1)\cdots \phi(n_r)}{n^2_1\cdots
n^2_r}= \sum_{n=1}^{\infty} \frac{\mu^{\times}(n)\phi(n)^r}{n^{2r}}
\end{equation*}
\begin{equation*}
= \prod_p \left(1-\frac{(p-1)^r}{p^r(p^r-1)}\right).
\end{equation*}
\end{proof}

Proposition \ref{Prop_unit_rel_prime} was proved by {\sc L.~T\'oth}
\cite{Tot2001} using different arguments. See \cite{Tot2001} for other
related densities and asymptotic formulas.

\begin{corollary} \label{Prop_unit_rel_prime_r_2} {\rm (r=2)} The set of points $(m,n)\in \N^2$ such that
$\gcud(m,n)=1$ has the density
\begin{equation*}
\prod_p \left(1-\frac{p-1}{p^2(p+1)}\right).
\end{equation*}
\end{corollary}

\begin{proposition} \label{Prop_pairw_unit_rel_prime} Let $r\ge 2$. The asymptotic density of the set of points in $\N^r$ with
pairwise unitary relatively prime coordinates is
\begin{equation*}
A^{\times}_r= \prod_p \sum_{\nu_1,\ldots,\nu_r=0}^{\infty}
\frac{Q(p^{\nu_1},\ldots, p^{\nu_r})\phi(p^{\nu_1})\cdots
\phi(p^{\nu_r})}{p^{2\nu_1+\ldots+ 2\nu_r}},
\end{equation*}
where $Q$ is the multiplicative function of $r$ variables given as
follows: Let $p^{\nu_1},\ldots,p^{\nu_r}$ be arbitrary powers of the
prime $p$ with $\nu_1,\ldots,\nu_r \in \N_0$, $\nu_1+\ldots
+\nu_r\ge 1$. Assume that the exponents $\nu_1,\ldots,\nu_r$ have
$q$ ($1\le q\le r$) distinct positive values, taken $t_1, \ldots,
t_q$ times ($1\le t_1+\ldots +t_q\le r$). Then
\begin{equation*}
Q(p^{\nu_1},\ldots,p^{\nu_r})= (-1)^r(1-t_1)\cdots (1-t_q).
\end{equation*}
\end{proposition}

\begin{proof} We use Proposition \ref{Prop_unitary_Wintner} for the function $\varrho^{\times}$
defined by \eqref{def_pair_unit_rel_prime}. The density is
\begin{equation*}
A^{\times}_r= \sum_{n_1,\ldots,n_r=1}^{\infty}
\frac{Q(n_1,\ldots,n_r)\phi(n_1)\cdots \phi(n_r)}{n^2_1\cdots
n^2_r},
\end{equation*}
where $Q_r=\mu^{\times}_r \times \varrho^{\times}$ and the Euler
product formula can be used by the multiplicativity of the involved
functions.
\end{proof}

Note that for $r=2$,
\begin{equation*}
Q(p^{\nu_1},p^{\nu_2})=
\begin{cases} \ \ 1, & \nu_1=\nu_2=0, \\ -1, & \nu_1=\nu_2\ge 1, \\ \ \ 0,  & \text{otherwise} \end{cases}
\end{equation*}
and reobtain Corollary \ref{Prop_unit_rel_prime_r_2}.

\begin{corollary} {\rm ($r=3,4$)}
\begin{equation*}
A^{\times}_3=\zeta(2)\zeta(3) \prod_p \left(1-\frac{4}{p^2} +
\frac{7}{p^3} - \frac{9}{p^4} +\frac{8}{p^5}- \frac{2}{p^6}
-\frac{3}{p^7}+ \frac{2}{p^8}\right).
\end{equation*}
\begin{equation*}
A^{\times}_4=\zeta^2(2)\zeta(3)\zeta(4) \prod_p
\left(1-\frac{8}{p^2} + \frac{3}{p^3} + \frac{27}{p^4}
-\frac{24}{p^5}- \frac{14}{p^6} -\frac{3}{p^7} \right.
\end{equation*}
\begin{equation*}
\left. + \frac{37}{p^8}- \frac{30}{p^9} + \frac{42}{p^{10}}- \frac{33}{p^{11}} -\frac{41}{p^{12}}+ \frac{78}{p^{13}}
-\frac{44}{p^{14}}+ \frac{9}{p^{15}} \right).
\end{equation*}
\end{corollary}

\begin{proof} According to Proposition \ref{Prop_pairw_unit_rel_prime},
\begin{equation*}
Q(p^{\nu_1},p^{\nu_2},p^{\nu_3})=
\begin{cases}  1, & \nu_1=\nu_2=\nu_3=0, \\  -1, & \nu_1=\nu_2\ge 1, \nu_3=0 \text{ and symmetric cases},
\\  2, & \nu_1=\nu_2=\nu_3 \ge 1,\\ 0,  & \text{otherwise} \end{cases}
\end{equation*}
and
\begin{equation*}
Q(p^{\nu_1},p^{\nu_2},p^{\nu_3},p^{\nu_4})=
\begin{cases}  1, & \nu_1=\nu_2=\nu_3=0, \\  -1, & \nu_1=\nu_2\ge 1, \nu_3=\nu_4=0  \text{ and
symmetric cases}, \\
-2, & \nu_1=\nu_2=\nu_3 \ge 1, \nu_4=0  \text{ and symmetric cases},
\\  -3, & \nu_1=\nu_2=\nu_3=\nu_4\ge 1,
\\  1, & \nu_1=\nu_2> \nu_3 =\nu_4 \ge 1 \text{ and
symmetric cases}, \\
0,  & \text{otherwise}, \end{cases}
\end{equation*}
and direct computations lead to the given infinite products.
\end{proof}

We refer to the papers by {\sc J.~Christopher} \cite{Chr1956}, {\sc
H.~Delange} \cite{Del1969} and {\sc N.~Ushiroya} \cite{Ush2012} for
related density results.

\subsection{Asymptotic formulas}

One has
\begin{equation} \label{(m,n)}
\sum_{m,n\le x} \gcd(m,n)= \frac{x^2}{\zeta(2)}\left(\log x+ 2\gamma
-\frac1{2}-\frac{\zeta(2)}{2}- \frac{\zeta'(2)}{\zeta(2)} \right) +
O(x^{1+\theta+\varepsilon}),
\end{equation}
for every $\varepsilon>0$, where $\theta$ is the exponent appearing
in Dirichlet's divisor problem, that is
\begin{equation} \label{Dirichlet_divisor}
\sum_{n\le x} \tau(n)= x\log x+(2\gamma-1)x+ O(x^{\theta+\varepsilon}).
\end{equation}

It is known that $1/4\le \theta \le 131/416 \approx 0.3149$, where
the upper bound, the best up to date, is the result of {\sc M.~N.~Huxley}
\cite{Hux2003}.

If $r\ge 3$, then
\begin{equation} \label{(n_1,...n_r)}
\sum_{n_1,\ldots,n_r\le x} \gcd(n_1,\ldots,n_r)=
\frac{\zeta(r-1)}{\zeta(r)} x^r + O(R_r(x)),
\end{equation}
where $R_3(x)=x^2\log x$ and $R_r(x)=x^{r-1}$ for $r\ge 4$, which
follows from the representation
\begin{equation*}
\sum_{n_1,\ldots,n_r\le x} \gcd(n_1,\ldots,n_r)=\sum_{d\le x} \phi(d)[x/d]^r.
\end{equation*}

Furthermore,
\begin{equation} \label{[m,n]}
\sum_{m,n\le x} \lcm(m,n)= \frac{\zeta(3)}{4\zeta(2)}x^4 + O(x^3\log x).
\end{equation}

The formulas \eqref{(m,n)}, \eqref{(n_1,...n_r)}, \eqref{[m,n]} can
be deduced by elementary arguments and go back to the work of {\sc
E.~Ces\`{a}ro} \cite{Ces1885}, {\sc E. Cohen} \cite{Coh1961iii} and
{\sc P.~Diaconis, P.~Erd\H os} \cite{DiaErd2004}. See also {\sc L. T\'oth} \cite[Eq.\
(25)]{Tot2010GCD}. A formula similar to \eqref{(m,n)}, with the same
error term holds for $\sum_{m,n\le x} g(\gcd(m,n))$, where
$g=h*\id$, $h\in {\cal A}_1$ is bounded, including the cases
$g=\phi,\sigma,\psi$. See {\sc L. T\'oth} \cite[p.\ 7]{Tot2010GCD}.
See {\sc J.~L.~Fern\'andez, P.~Fern\'andez} \cite{FF2013PR,FF2013NT,FF2013NToct} for statistical regularity
properties of the gcd's and lcm's of positive integers.

Consider next the  function $g_2(n)=\sum_{d\mid n} \gcd(d,n/d)$,
which is the Dirichlet convolute of the gcd function for $r=2$.

\begin{proposition}
\begin{equation*}
\sum_{n\le x} g_2(n)= \frac{3}{2\pi^2} x (\log^2 x + c_1\log x+ c_2) +
R(x),
\end{equation*}
where $c_1,c_2$ are constants and $R(x)= O(x^{\theta}
(\log x)^{\theta'})$ with $\theta=\frac{547}{832}=0.65745\dots$,
$\theta'=\frac{26947}{8320}$.
\end{proposition}

This was proved using analytic tools (Huxley's method) by {\sc
E.~Kr\"atzel, W.~G. Nowak, L.~T\'oth} \cite[Th.\
3.5]{KraNowTot2012}. See {\sc M.~ K\"uhleitner, W.~G. Nowak}
\cite{KuhNow2013} for omega estimates on the function $g_2(n)$. The
papers \cite{KraNowTot2012} and \cite{KuhNow2013} contain also
results for the function $g_r(n)$ ($r\ge 3$), defined by
\eqref{def_g_r} and related functions.

For the function $\ell_2(n)=\sum_{d\mid n} \lcm(d,n/d)$, representing the
Dirichlet convolute of the lcm function for $r=2$ one can deduce the next asymptotics.

\begin{proposition}
\begin{equation*}
\sum_{n\le x} \ell_2(n)/n= \frac{\zeta(3)}{\zeta(2)}x \left(\log x+
2\gamma-1 -\frac{2\zeta'(2)}{\zeta(2)} +\frac{2\zeta'(3)}{\zeta(3)}
\right)+ O(x^{\theta+\varepsilon}),
\end{equation*}
where $\theta$ is given by \eqref{Dirichlet_divisor}.
\end{proposition}

A similar formula can be given for the function $\ell_r(n)$ ($r\ge 3$) defined by \eqref{def_ell_r}.

For the functions $s(m,n)$ and $c(m,n)$ defined in Section
\ref{subsection_examples}, Example \ref{ex_s_c}, {\sc W.~G. Nowak,
L.~T\'oth} \cite{NowTot2013} proved the following asymptotic
formulas.

\begin{proposition} For every fixed $\varepsilon>0$,
\begin{equation*}
\sum_{m,n\le x} s(m,n) = \frac{2}{\pi^2} x^2(\log^3 x+a_1\log^2
x+a_2\log x + a_3)+ O\left({x^{\frac{1117}{701}+\varepsilon}}
\right),
\end{equation*}
\begin{equation*}
\sum_{m,n\le x} c(m,n) = \frac{12}{\pi^4} x^2(\log^3 x+b_1\log^2
x+b_2\log x + b_3)+ O\left({x^{\frac{1117}{701}+\varepsilon}}
\right),
\end{equation*}
where $1117/701\approx 1.5934$ and $a_1,a_2,a_3,b_1,b_2,b_3$ are
explicit constants.
\end{proposition}

See the recent paper by {\sc T.~H.~Chan, A.~V.~Kumchev}
\cite{ChaKum2012} concerning asymptotic formulas for $\sum_{n\le x,
k\le y} c_n(k)$.

\section{Acknowledgement} The author gratefully acknowledges support
from the Austrian Science Fund (FWF) under the project Nr.
M1376-N18.


\medskip

\noindent L. T\'oth \\
Institute of Mathematics, Universit\"at f\"ur Bodenkultur \\
Gregor Mendel-Stra{\ss}e 33, A-1180 Vienna, Austria \\ and \\
Department of Mathematics, University of P\'ecs \\ Ifj\'us\'ag u. 6,
H-7624 P\'ecs, Hungary \\ E-mail: ltoth@gamma.ttk.pte.hu

\end{document}